\documentclass[leqno]{siamltex704}
\usepackage{amssymb,amsmath,graphicx,amscd,mathrsfs}
\usepackage{color,xcolor,amsmath}
\usepackage{amsmath}
\usepackage{graphicx}
\usepackage{mathrsfs}
\usepackage{float}
\usepackage{amsfonts,amssymb}
\usepackage{dsfont}
\usepackage{pifont}
\usepackage{hyperref}
\usepackage{multirow}
\numberwithin{equation}{section}
\def\3bar{{|\hspace{-.02in}|\hspace{-.02in}|}}
\def\E{{\mathcal{E}}}
\def\T{{\mathcal{T}}}

\def\bQ{{\mathbf{Q}}}

\def\b0{\boldsymbol{0}}

\def\bw{{\mathbf{w}}}
\def\bu{{\mathbf{u}}}
\def\bv{{\mathbf{v}}}
\def\bn{{\mathbf{n}}}
\def\be{{\mathbf{e}}}
\def\bf{{\mathbf{f}}}

\def\bg{{\mathbf{g}}}

\newtheorem{algorithm1}{Weak Galerkin Algorithm}

\title{A Weak Galerkin Finite Element Scheme for solving the stationary Stokes Equations}
\author{ Ruishu Wang\thanks{Department of Mathematics, Jilin University, Changchun,
China (ruishu@email.jlu.edu.cn). } \and Xiaoshen
Wang\thanks{Department of Mathematics, University of Arkansas at
Little Rock, Little Rock, AR 72204, United States(xxwang@ualr.edu).
} \and Qilong Zhai
\thanks{Department of Mathematics, Jilin University, Changchun,
China (diql13@mails.jlu.edu.cn). } \and Ran Zhang\thanks{Department
of Mathematics, Jilin University, Changchun, China
(zhangran@mail.jlu.edu.cn). The research of Zhang was supported in
part by China Natural National Science Foundation(11271157,
11371171, 11471141), and by the Program for New Century Excellent
Talents in University of Ministry of Education of China.}
 }

\begin{document}

\maketitle

\begin{abstract}
A weak Galerkin (WG) finite element method for solving the
stationary Stokes equations in {\color {black}two- or three- dimensional} spaces by
using discontinuous piecewise polynomials is {\color {black}developed} and analyzed.  The variational form
we considered is based on two gradient operators which is different
from the usual gradient-divergence operators. The WG method is
highly flexible by allowing the use of discontinuous functions on
arbitrary polygons or polyhedra with certain shape regularity.
Optimal-order error estimates are established for the corresponding
WG finite element solutions in various norms.{\color{black}
Numerical results are presented to illustrate the theoretical analysis
of the new WG finite element scheme for Stokes problems.}
\end{abstract}

\begin{keywords} weak Galerkin finite element methods,
weak gradient, Stokes equations, polytopal meshes.
\end{keywords}

\begin{AMS}
Primary, 65N30, 65N15, 65N12, 74N20; Secondary, 35B45, 35J50, 35J35
\end{AMS}

\section{Introduction}

The aim of this paper is to present a novel {\color {black} weak} Galerkin finite
element method for solving the stationary Stokes {\color {black} equations}. Let $\Omega$ be a polygonal or
polyhedral domain in $\mathbb{R}^d, d=2,3$. As a model for the flow
of an incompressible viscous fluid confined in $\Omega$, we consider
the following equations
\begin{eqnarray}
\label{ose1}-\mu\Delta \textbf{u}+\nabla p &=& \textbf{f}, \quad
{\rm in}\ \Omega,
\\
\label{ose2}\nabla \cdot\textbf{u}&=&0,\quad {\rm in}\ \Omega,
\\
\label{ose3}\textbf{u}&=&\textbf{g},\quad {\rm on}\ \partial\Omega,
\end{eqnarray}
for unknown velocity function $\bu$ and pressure function $p$ (we
require that $p$ has zero average in order to guarantee the
uniqueness of the pressure). Bold symbols are used to denote vector-
or tensor-valued functions or spaces of such functions. Here $\bf$
is a {\color {black} body source} term, $\mu>0$ is the kinematic viscosity and
$\bg$ is a boundary condition that satisfies the compatibility
condition
$$
\int_{\partial \Omega} \bg \cdot \bn~ ds=0,
$$
where $\bn$ is the unit outward normal vector on the domain boundary
$\partial \Omega$.

This problem mainly arises from approximations of
low-Reynolds-number flows. The finite element methods for Stokes and
Navier$-$Stokes problems enforce the {\color {black}\\divergence-free} property in
finite element spaces, which satisfy the inf-sup (LBB) condition, in
order for them to be numerically stable
\cite{BS08,babuska,GR86,Gunzburger89, ERS07}.
{\color{black}
The Stokes problem has been studied with various different new numerical methods:
\cite{CG05,KK06,Liu11,WangWangYe09,YangLiuLin15}.
}

Throughout this paper, we would follow the standard definitions for
Lebesgue and Sobolev spaces: $L^2(\Omega)$, $H^1(\Omega)$,
$[L^2(\Omega)]^d$,
$$
[H_0^1(\Omega)]^d=\{\bv\in [H^1(\Omega)]^d: \bv=\b0\ {\rm on}\
\partial \Omega\} $$ and
$$
L_0^2(\Omega):=\{q\in
L^2(\Omega): \int_{\Omega} q dx=0\}
$$
are the natural spaces for the weak form of the Stokes problem
\cite{GR86,BF91}. Denote $(\cdot,\cdot)$ for inner products in the
corresponding spaces.

{\color{black}Next we assume that $\mu=1$ and $\textbf{g}=\b0$. Then
one of the variational formulations for the Stokes problem
(\ref{ose1})-(\ref{ose3}) is to find $\bu \in [H_0^1(\Omega)]^d$ and
$p\in L_0^2(\Omega)$ such that
\begin{eqnarray}
\label{VF1}(\nabla\bu,\nabla\bv)-(\nabla\cdot\bv, p)&=&(\bf,\bv),
\\
\label{VF2}(\nabla\cdot \bu, q)&=&0,
\end{eqnarray}
for all $\bv\in[H_0^1(\Omega)]^d$ and $q\in L_0^2(\Omega)$. Here
$\nabla\bu$ denotes the velocity gradient tensor
$(\nabla\bu)_{ij}=\partial_j \bu_i$. It is well known that under our
assumptions on the domain and the data, problem
(\ref{VF1})-(\ref{VF2}) has a unique solution $(\bu;p)\in
[H_0^1(\Omega)]^d\times L_0^2(\Omega)$.}

{\color {black}{
For any $p \in L^2_0(\Omega)$, define a functional $\nabla p$ such that
\begin{eqnarray*}
\langle \nabla p,\bv\rangle=-(\nabla\cdot\bv,p),\quad \forall \bv\in [H_0^1(\Omega)]^d.
\end{eqnarray*}
}}
It is easy to know that the weak form (\ref{VF1})-(\ref{VF2}) is
also equivalent to the following variational problem: find
$(\bu;p)\in [H_0^1(\Omega)]^d\times L_0^2(\Omega)$ such that
\begin{eqnarray}
\label{VF3}(\nabla\bu,\nabla\bv)+\langle\nabla p,\bv\rangle&=&(\bf,\bv),
\\
\label{VF4}\langle\nabla q,\bu\rangle&=&0,
\end{eqnarray}
for all $\bv\in[H_0^1(\Omega)]^d$ and $q\in L_0^2(\Omega)$. The
unique solvability of (\ref{VF3})-(\ref{VF4}) follows directly from
that of the (\ref{VF1})-(\ref{VF2}).

The WG method refers to a general finite element technique for
partial differential equations where differential operators are
approximated as distributions for generalized functions. This method
was first proposed in \cite{WangYe2013,WangYe2014,mwy0927} for
second order elliptic problem, then extended to other partial
differential equations \cite{maxwell, mwyz-biharmonic, first,
MWY14JCP, ZZW15, ZZ15}. Weak functions and weak derivatives can be
approximated by polynomials with various degrees. The WG method uses
weak functions and their weak derivatives which are defined as
distributions. The most prominent features of it are:

\begin{itemize}
\item {The usual derivatives are replaced by distributions or discrete approximations of distributions.}
\item {The approximating functions are discontinuous. The flexibility of discontinuous functions gives WG
methods many advantages, such as high order of accuracy, high
parallelizability, localizability, and easy handling of complicated
geometries.}
\end{itemize}

The above features motivate the use of WG methods for the Stokes
equations. It can easily handle meshes with hanging nodes, elements
of general shapes with certain shape regularity and ideally suited
for hp-adaptivity. {\color{black} In \cite{WangYeStokes}, Wang et.
al. considered WG methods for {\color {black}{the Stokes equations}}
(\ref{VF1})-(\ref{VF2}). Similarly, in \cite{MWY14JCP}, they
presented WG methods for the Brinkman equations, which is a model
with a high-contrast parameter dependent combination of the Darcy
and Stokes models. The numerical method of  \cite{MWY14JCP} is based
on the traditional gradient-divergence variational form for the
Brinkman equations. In \cite{ZhaiZhangMu}, we presented a new WG
scheme based on the gradient-gradient variational form. It is shown
that this scheme is suit for the mixed formulation of Darcy which
would present a better approximation for this case. In fact, for
complex porous media with interface conditions, people often use
Brinkman-Stokes interface model to describe this problem, which is
an ongoing work for us now. In order to present a more efficient WG
scheme, we prefer to utilize this gradient-gradient weak form to
approximate the model. In order to unify the weak form of this
interface problem, we need the numerical analysis results of this
form for Stokes problem. However, to the best of our knowledge, the
numerical analysis of methods based on the variational form
(\ref{VF3})-(\ref{VF4}) has never been done before. Therefore in
this paper, we propose a WG method based on the weak form
(\ref{VF3})-(\ref{VF4}) of the primary problem. In addition, if we
choose high order polynomials to approximate the model and use Schur
complement to reduce the interior DOF of the velocity and pressure
by the boundary DOF, the total DOF of this new method is less than
the scheme of \cite{WangYeStokes}.}

The rest of this paper is organized as follows. In Section 2 we
shall introduce some preliminaries and notations for Sobolev spaces.
Section 3 is devoted to the definitions of weak functions and weak
derivatives. The WG finite element schemes for variational form of
the Stokes equation (\ref{VF3})-(\ref{VF4}) are presented in Section
4. This section also contains some local $L^2$ projection operators
and then derives some approximation properties which are useful in a
convergence analysis. In Section 5, we derive an error equation for
the WG finite element approximation. Optimal-order error estimates
for the WG finite element approximations are derived in Section 6 in
an $H^1$-equivalent norm for the velocity, and $L^2$ norm for both
the velocity and the pressure.  In Section 7, we present some
numerical results which confirm the theory developed in earlier
sections. Finally, we present some technical estimates in the
appendix for quantities related to the local $L^2$ projections into
various finite element spaces.

\section{Preliminaries and Notations}

Let $K\subset\Omega$ be an open bounded domain with Lipschitz
continuous boundary in $\mathbb{R}^d, d=2, 3$. We shall use standard
definitions of the Sobolev spaces $H^s(K)$ and inner products
$(\cdot, \cdot)_{s, K}$, their norms $\|\cdot\|_{s, K}$, and
seminorms $|\cdot|_{s, K}$, for any $s\ge 0$. For instance, for any
integer $s\ge 0$, the seminorm $|\cdot|_{s, K}$ is defined as
$$
|v|_{s, K}=\left(\sum_{|\alpha|=s}\int_K |\partial^\alpha v|^2 {\rm
d} K\right)^{\frac{1}{2}},
$$
with notations
$$
\alpha=(\alpha_1,\alpha_2, \cdots, \alpha_d), \quad
|\alpha|=\alpha_1+\alpha_2+\cdots+\alpha_d, \quad
\partial^\alpha=\prod_{j=1}^d \partial_{x_j}^{\alpha_j}.
$$
The Sobolev norm $\|\cdot\|_{m,K}$ is defined as
$$
\|v\|_{m, K}=\left(\sum_{j=0}^{m} |v|_{j, K}^2\right)^{\frac{1}{2}}.
$$

The space $H^0(K)$ is same as $L^2(K)$, whose norm and inner product are denoted by $\|\cdot\|_K$ and
$(\cdot,\cdot)_K$, respectively. If $K=\Omega$, we would drop the
subscript $K$ in the notations of the $L^2$ norm and the $L^2$ inner product.


\section{Weak Differential Operators}

In this section we will define weak functions for both the
vector-valued function and the scalar-valued function, also we will
introduce the weak gradients and the corresponding discrete forms.
\subsection{Weak gradient for weak vector-valued function}

Let $T$ be a polygonal or polyhedral domain with boundary $\partial T$.
A weak vector-valued function on the domain $T$ is defined as $\bv = \{\bv _0,\bv _b\}$
such that $\bv _0 \in [ L^2(T)]^d$ and $\bv_b \in [L^2(\partial T)]^d$. Let
\begin{eqnarray*}
V(T)=\{ \bv =\{ \bv _0 ,\bv_b \}; \bv_0 \in [L^2(T)]^d,\bv _b \in [L^2(\partial T)]^d \},
\end{eqnarray*}
where $\bv_b$ is not necessarily the trace of $\bv_0$.
\begin{definition}{\rm(\cite{WangYeStokes})}
For any $\bv \in V(T)$, the weak gradient of $\bv$, denoted by $\nabla_w\bv$, is defined as  a linear functional in the dual space of $[H^1(T)]^{d\times d}$ whose action on each $\tau\in [H^1(T)]^{d\times d} $ is given by
\begin{eqnarray*}
(\nabla_w\bv,\tau)_T=-(\bv_0,\nabla\cdot \tau)_T+\langle\bv_b,\tau\cdot \bn\rangle_{\partial T},~~~\forall \tau\in [H^1(T)]^{d\times d},
\end{eqnarray*}
where $\bn$ is the outer unit normal vector to $\partial T$, $(\bv_0,\nabla\cdot \tau)_T$ is the $L^2$ inner product of $\bv_0$ and $\nabla\cdot\tau$, and $\langle\bv_b,\tau\cdot \bn\rangle_{\partial T}$ is the inner product of  $\tau\cdot\bn$ and $\bv_b$ in $[L^2(\partial T)]^d$.
\end{definition}

Consider the inclusion map $i_V : [H^1(T)]^d\rightarrow V(T)$
defined below
  $$
     i_V (\phi)=\{\phi|_T,\phi|_{\partial T}\},  ~~~~ \phi\in[H^1(T)]^d.
  $$

By this map the Sobolev space $[H^1(T)]^d$ can be embedded into the space $V(T)$. With the help of map $i_V$, the Sobolev space $[H^1(T)]^d$ can be considered as a subspace of $V(T)$ by identifying each $\phi\in[H^1(T)]^d$ with $i_V(\phi)$.

Let $P_r(T)$ be the set of polynomials on T with degree no more than $r$.
\begin{definition}{\rm(\cite{WangYeStokes})}
The discrete weak gradient operator $\nabla_{w,r,T}$ is defined as
follows: for each {$\bv\in V(T)$}, $\nabla_{w,r,T}\bv
\in[P_r(T)]^{d\times d}$ is the unique element such that
\begin{eqnarray}\label{dwvg}
(\nabla_{w,r,T}\bv,\tau)_T=-(\bv_0,\nabla\cdot \tau)_T+\langle\bv_b,\tau\cdot \bn\rangle_{\partial T},~~~~\forall \tau\in[P_r(T)]^{d\times d}.
\end{eqnarray}
\end{definition}

\subsection{Weak gradient for weak scalar-valued function}
We define a weak scalar-valued function on the domain $T$ as $q = \{q _0,q _b\}$ such that $q _0 \in L^2(T)$ and $q_b \in L^2(\partial T)$. Let
\begin{eqnarray*}
W(T)=\{q =\{q _0 ,q_b \}; q_0 \in L^2(T),q _b \in L^2(\partial T)\},
\end{eqnarray*}
where $q_b$ is not necessarily the trace of $q_0$.

\begin{definition}{{\rm(\cite{WangYe2013})}}
For any $q \in W(T)$, the weak gradient of $q$, denote by $\widetilde{\nabla}_wq$, is defined as  a linear functional in the dual space of $[H^1(T)]^2$ whose action on each $\bw\in [H^1(T)]^2 $ is given by
\begin{eqnarray*}
(\widetilde{\nabla}_w q,\bw)_T=-(q_0,\nabla\cdot\bw)_T+\langle q_b,\bw\cdot \bn\rangle_{\partial T},~~~\forall \bw\in [H^1(T)]^2,
\end{eqnarray*}
where $\bn$ is the outer unit normal vector to $\partial T$, $(q_0,\nabla\cdot\bw)_T$ is the $L^2$ inner product of $q_0$ and $\nabla\cdot\bw$, and $\langle q_b,\bw\cdot \bn\rangle_{\partial T}$ is the inner product of  $\bw\cdot \bn$ and  $q_b$ in $L^2(\partial T)$.
\end{definition}

Consider the inclusion map $i_W : H^1(T)\rightarrow W(T)$ defined as follows
  $$
     i_W (\phi)=\{\phi|_T,\phi|_{\partial T}\},  ~~~~ \phi\in H^1(T).
  $$
By which the Sobolev space $ H^1(T)$ is embedded into the space
$W(T)$.  With the help of map $i_W$, the Sobolev space $H^1(T)$ can
be considered as a subspace of $W(T)$ by identifying each $\phi\in
H^1(T)$ with $i_W(\phi)$.
\begin{definition}{{\rm(\cite{WangYe2013})}}
The discrete weak gradient operator $\widetilde{\nabla}_{w,r,T}$ is
defined as follows: for each {$q\in W(T)$},
$\widetilde{\nabla}_{w,r,T}q \in[P_r(T)]^d$ is the unique element
such that
\begin{eqnarray}\label{dwqg}
(\widetilde{\nabla}_{w,r,T}q,\bw)_T=-(q_0,\nabla\cdot\bw)_T+\langle q_b,\bw\cdot \bn\rangle_{\partial T},~~~~\forall \bw\in[P_r(T)]^d.
\end{eqnarray}
\end{definition}

\section{A Weak Galerkin Finite Element Scheme}

Let $\mathcal{T}_h$ be a partition of the domain $\Omega$ into
polygons in 2D or polyhedral in 3D. Assume that $\mathcal{T}_h$ is
shape regular in the sense as defined in {\rm\cite{first}}.
Denote by $\mathcal{E}_h$ the set of all edges or flat faces in
$\mathcal{T}_h$, and let
$\mathcal{E}_h^0=\mathcal{E}_h\setminus\partial \Omega$ be the set
of all interior edges or flat faces. Denote by $h_T$ the diameter of
$T\in \mathcal{T}_h$ and $h=\max_{T\in \mathcal{T}_h}h_T$ the
meshsize for the partition $\mathcal{T}_h$.

For any interger $k\ge 1$, we define weak Galerkin finite element
spaces as follows: for velocity variable, let
$$
V_h=\{ \bv=\{\bv_0,\bv_b\}; \{\bv_0,\bv_b\} |_T\in[P_k(T)]^d\times
[P_k(e)]^d,e\subset \partial T,\bv_b=\b0  ~\rm{on}~  \partial
\Omega\}.
$$

It should be noticed that $\bv_b$ is single valued on each edge $e\subset \mathcal{E}_h$. For pressure variable, we define
$$
W_h=\left\{ q=\{q_0,q_b\}; \sum_{T\in \mathcal{T}_h}\int_T q_0 dT=0, \{q_0,q_b\} |_T\in P_{k-1}(T)\times P_k(e), e\subset \partial T\right\}.
$$
Also $q_b$ is single valued on each edge $e\subset \mathcal{E}_h$.

The discrete weak gradients $\nabla_{w,k-1}$ and $\widetilde{\nabla}_{w,k}$ on the spaces $V_h$ and $W_h$ can be computed by the equations (\ref{dwvg}) and (\ref{dwqg}) on each element $T$ respectively, that is,
\begin{eqnarray*}
(\nabla_{w,k-1}\bv)|_T&=&\nabla_{w,k-1,T}(\bv|_T), \quad{\forall \bv \in V_h},
\\
(\widetilde{\nabla}_{w,k}q)|_T&=&\widetilde{\nabla}_{w,k,T}(q|_T),~~~\quad{\forall q \in W_h}.
\end{eqnarray*}
For the sake of simplicity, we shall drop the subscripts $k-1$ and $k$ of $\nabla_{w,k-1}$ and $\widetilde{\nabla}_{w,k}$ in the rest of the paper.

We use the $L^2$ inner product to denote the sum of inner products on each of the elements as follows:
\begin{eqnarray*}
(\nabla_w\bv,\nabla_w\bw)&=&\sum_{T\in\T_h}(\nabla_w\bv,\nabla_w\bw)_T,
\\
(\widetilde{\nabla}_wq,\bv)&=&\sum_{T\in T_h}(\widetilde{\nabla}_wq,\bv)_T.
\end{eqnarray*}

\begin{lemma}{{\rm(\cite{WangYeStokes})}}
For any $\bv\in V_h$ and $p\in W_h$ the following equations hold true
\begin{eqnarray}
\label{ifw1}(\nabla_w \bv,\tau)_T&=&(\nabla\bv_0,\tau)_T-\langle\bv_0-\bv_b,\tau\cdot\bn\rangle_{\partial T},~~~~\forall \tau\in[P_{k-1}(T)]^{d\times d},
\\
\label{ifw2}(\widetilde{\nabla}_w p,\bw)_T&=&(\nabla p_0,\bw)_T-\langle p_0-p_b,\bw\cdot\bn\rangle_{\partial T},~~~\forall \bw\in[P_k(T)]^d.
\end{eqnarray}
\end{lemma}

For each element $T\in \T_h$, denote by $Q_0$ the $L^2$ projection
operator from $[L^2(T)]^d$ onto $[P_k(T)]^d$. For each edge or face
$e\in \E_h$, denote by $Q_b$ the $L^2$ projection from $[L^2(e)]^d$
onto $[P_{k}(e)]^d$.We shall combine $Q_0$ with $Q_b$ as a projection onto $V_h$, such that
on each element $T\in\mathcal{T}_h$
\begin{eqnarray*}
Q_h\bu=\{Q_0\bu,Q_b\bu\}.
\end{eqnarray*}
{\color {black}On each element $T\in \T_h$, denote by $\bQ_h$ the
$L^2$ projection onto $[P_{k-1}(T)]^{d\times d}$. Denote by
$\widetilde{Q}_0$ the $L^2$ projection operator from $L^2(T)$ onto
$P_{k-1}(T)$. For each edge or face $e\in \E_h$, denote by
$\widetilde{Q}_b$ the $L^2$ projection from $L^2(e)$ onto
$P_{k}(e)$.  {We shall combine $\widetilde{Q}_0$ with
$\widetilde{Q}_b$ as a projection onto space $W_h$, such that on
each element $T\in\mathcal{T}_h$
\begin{eqnarray*}
\widetilde{Q}_h q=\{\widetilde{Q}_0 q, \widetilde{Q}_b q\}.
\end{eqnarray*}}}

Then we shall present a useful property which indicates the discrete weak gradient operators are good approximation to the gradient operators in the classical sense.
\begin{lemma}{{\rm(\cite{WangYeStokes})}}
The following equations hold true.
\begin{eqnarray}
\label{qcfv}\nabla_w Q_h\bv&=&\bQ_h\nabla\bv,~~~\forall\bv\in [H^1(\Omega)]^d,
\\
\label{qcfp}\widetilde{\nabla}_w\widetilde{Q}_hp&=&Q_0\nabla p,~~~~\forall p \in H^1(\Omega).
\end{eqnarray}
\end{lemma}
%

Now we introduce four bilinear forms as follows:
\begin{eqnarray}
\label{bf2}s(\bw,\bv)&=&\sum_{T\in\mathcal{T}_h}h_T^{-1}\langle \bw_0-\bw _b,\bv_0-\bv_b\rangle_{\partial T},
\\
\label{bf1}a(\bw,\bv)&=&(\nabla_w\bw,\nabla_w\bv)+s(\bw,\bv),
\\
\label{bf3}b(\bw,q)&=&(\bw_0,\widetilde{\nabla}_w q),
\\
\label{bf4}c(\rho,q)&=&\sum_{T\in\mathcal{T}_h}h_T\langle \rho_0-\rho_b,q_0-q_b\rangle_{\partial T}.
\end{eqnarray}
Using these bilinear forms we define the following two norms. For any $\bv\in V_h^0 $ and $q\in W_h$,
\begin{eqnarray}
\label{wn1}\3bar\bv\3bar^2&=&a(\bv,\bv)=(\nabla_w\bv,\nabla_w\bv)+\sum_{T\in\mathcal{T}_h}h_T^{-1}\langle \bv_0-\bv _b,\bv_0-\bv_b\rangle_{\partial T},
\end{eqnarray}
and
\begin{eqnarray}
\label{wn3}\3bar q\3bar_0^2&=&\|q_0\|^2+\3bar q\3bar_*^2,
\end{eqnarray}
where $\3bar q\3bar^2_*=c(q,q)$ is a seminorm.

It is easy to verify that $\3bar\cdot\3bar$ and $\3bar\cdot\3bar_0$ are norms in $V_h$ and $W_h$, respectively,
\begin{algorithm1}
A numerical approximation for (\ref{ose1})-(\ref{ose3}) can be
obtained by seeking $\bu_h=\{\bu_0,\bu_b\}\in V_h$ and
$p_h=\{p_0,p_b\}\in W_h$ such that
\begin{eqnarray}
\label{wf1}a(\bu_h,\bv)+b(\bv,p_h)&=&(\bf,\bv_0),
\\
\label{wf2}b(\bu_h,q)-c(p_h,q)&=&0,
\end{eqnarray}
for all $\bv=\{\bv_0,\bv_b\}\in V_h$ and $q\in W_h$.
\end{algorithm1}

Next we shall show that the weak Galerkin finite element algorithm
(\ref{wf1})-(\ref{wf2}) has only one solution. Since the system is
linear, it suffices to show that if $\bf=\textbf{0}$, the only
solution is $\bu_h=\{\b0,\b0\}; p_h=\{0,0\}$.
\begin{lemma}
The WG finite element scheme (\ref{wf1})-(\ref{wf2}) has a unique solution.
\end{lemma}
\begin{proof}
Let $\bf=\textbf{0}$, we shall show that the solution of
(\ref{wf1})-(\ref{wf2}) is trivial. To this end, taking $\bv=\bu_h$
and $q=p_h$ and subtracting (\ref{wf2}) from (\ref{wf1}) we arrive
at
$$
a(\bu_h,\bu_h)+c(p_h,p_h)=0.
$$
By the definition of $a(\cdot,\cdot)$ and $c(\cdot,\cdot)$, we know
$\nabla_w\bu_h=0$ on each $T\in\mathcal{T}_h$, $ \bu_0=\bu_b$, and
$p_0=p_b$ on each $\partial T$. Thus $\bu_0$ and $q_0$ are
continuous.

By (\ref{ifw1}) and the fact that $\bu_b=\bu_0$ on $\partial T$ we have, for any $\tau\in[P_{k-1}(T)]^{d\times d}$,
\begin{eqnarray*}
0&=&(\nabla_w\bu_h,\tau)_T
\\
&=&(\nabla\bu_0,\tau)_T-\langle\bu_0-\bu_b,\tau\cdot\bn\rangle_{\partial T}
\\
&=&(\nabla\bu_0,\tau)_T,
\end{eqnarray*}
which implies $\nabla\bu_0=0$ on each $T\in\mathcal{T}_h$ and thus
$\bu_0$ is a constant. Since $\bu_0=\bu_b$ on each $\partial T$ and
$\bu_b=\textbf0$ on $\partial\Omega$, we arrive at
$\bu_h=\{\b0,\b0\}$ in $\Omega$. It follows from (\ref{wf1}),
$\bu_h=\{\b0,\b0\}$, and $\bf=\b0$ that for any $\bv\in V_h$,
\begin{eqnarray*}
0&=&b(\bv,p_h)
\\
&=&(\bv_0,\widetilde{\nabla}_wp_h)
\\
&=&\sum_{T\in\mathcal{T}_h}(\bv_0,\nabla p_0)_T-\sum_{T\in\mathcal{T}_h}\langle\bv_0\cdot\bn,p_0-p_b\rangle_{\partial T}
\\
&=&\sum_{T\in\mathcal{T}_h}(\bv_0,\nabla p_0)_T.
\end{eqnarray*}
Hence we have $\nabla p_0=0$ on each $T\in\mathcal{T}_h$. Thus $p_0$
is a constant in $\Omega$. From $p_0\in L_0^2(\Omega)$, we would
obtain $p_0=0$ in $\Omega$. Since $p_b=p_0$ on each $\partial T$,
$p_b=0$.

This completes the proof of the lemma.
\end{proof}

\section{Error Equation}
In this section, we shall derive the error equations for the WG finite element solution we get from (\ref{wf1})-(\ref{wf2}). This error equation is essential for the following analysis.

Now we define two bilinear forms
\begin{eqnarray}
\label{poe1}l_1(\bw,\bv)&=&\sum_{T\in\mathcal{T}_h}\langle\bv_0-\bv_b,(\nabla\bw-\bQ_h\nabla\bw)\cdot\bn\rangle_{\partial T},
\\
\label{poe3}l_2(\bw,q)&=&\sum_{T\in\mathcal{T}_h}\langle q_0-q_b,(\bw-Q_0\bw)\cdot\bn\rangle_{\partial T},
\end{eqnarray}
for all $\bw\in[H^1(\Omega)]^d, \bv\in V_h$ and $q\in W_h$.

Let $(\bu;p)$ be the exact solution of (\ref{ose1})-(\ref{ose3}), and $(\bu_h;p_h)\in V_h\times W_h$ be the solution of (\ref{wf1})-(\ref{wf2}).

Define
$$
\be_h=Q_h\bu-\bu_h,~~~ \varepsilon_h=\widetilde{Q}_hp-p_h.
$$
We shall derive the error equations that $\be_h\in V_h$ and $\varepsilon_h\in W_h$ satisfy.
\begin{lemma}
Let $\bu_h\in V_h$ and $p_h\in W_h$ be the solution of the numerical scheme (\ref{wf1})-(\ref{wf2}), and $(\bu;p)$ be the exact solution of (\ref{ose1})-(\ref{ose3}). Then, for any $\bv\in V_h$ and $q\in W_h$ we have
\end{lemma}
\begin{eqnarray}
\label{ee1}a(\be_h,\bv)+b(\bv,\varepsilon_h)&=&s(Q_h\bu,\bv)+l_1(\bu,\bv),
\\
\label{ee2}b(\be_h,q)-c(\varepsilon_h,q)&=&l_2(\bu, q)-c(\widetilde{Q}_hp,q).
\end{eqnarray}
\begin{proof}
First, from (\ref{ifw1}) and the property (\ref{qcfv}) we obtain
\begin{equation*}
\begin{aligned}
(\nabla_wQ_h\bu,\nabla_w\bv)_T=&(\bQ_h\nabla\bu,\nabla_w\bv)_T
\\
=&(\nabla\bv_0,\bQ_h\nabla\bu)_T-\langle\bv_0-\bv_b,(\bQ_h\nabla\bu)\cdot\bn\rangle_{\partial T}
\\
=&(\nabla\bv_0,\nabla\bu)_T-\langle\bv_0-\bv_b,(\bQ_h\nabla\bu)\cdot\bn\rangle_{\partial T}.
\end{aligned}
\end{equation*}
Summing over all elements $T\in\mathcal{T}_h$, we have
\begin{eqnarray}\label{eefa1}
(\nabla_w Q_h\bu,\nabla_w\bv)=(\nabla\bu,\nabla\bv_0)-\sum_{T\in\mathcal{T}_h}\langle\bv_0-\bv_b,(\bQ_h\nabla\bu)\cdot\bn\rangle_{\partial T}.
\end{eqnarray}
From the commutative property (\ref{qcfp}) we arrive at
\begin{eqnarray}\label{eefb1}
b(\bv,\widetilde{Q}_hp)=(\bv_0,\widetilde{\nabla}_w\widetilde{Q}_hp)=(\bv_0,\nabla p).
\end{eqnarray}
It follows from (\ref{ifw2}) that
\begin{equation*}
\begin{aligned}
(Q_0\bu,\widetilde{\nabla}_wq)_T=&(\nabla q_0,Q_0\bu)_T-\langle q_0-q_b,Q_0\bu\cdot\bn\rangle_{\partial T}
\\
=&(\nabla q_0,\bu)_T-\langle q_0-q_b,Q_0\bu\cdot\bn\rangle_{\partial T}.
\end{aligned}
\end{equation*}
Summing over all $T\in\mathcal{T}_h$ yields
\begin{equation}\label{eefb2}
\begin{aligned}
b(Q_h\bu,q)&=(Q_0\bu,\widetilde{\nabla}_w q)
\\
&=(\bu,\nabla q_0)-\sum_{T\in\mathcal{T}_h}\langle q_0-q_b,Q_0\bu\cdot\bn\rangle_{\partial T}.
\end{aligned}
\end{equation}

Next, using $\bv_0$ in $\bv=\{\bv_0,\bv_b\}\in V_h$ to test (\ref{ose1}), we have
\begin{eqnarray*}
(-\Delta \bu,\bv_0)+(\bv_0, \nabla p)&=&(\bf,\bv_0).
\end{eqnarray*}
Integrating by parts, we obtain
\begin{eqnarray}
\label{eefa2}(\nabla\bu,\nabla\bv_0)+(\bv_0, \nabla p)&=&(\bf,\bv_0)+\sum_{T\in\mathcal{T}_h}\langle\nabla\bu\cdot\bn,\bv_0-\bv_b\rangle_{\partial T},
\end{eqnarray}
where we have used the fact that $\sum_{T\in\mathcal{T}_h}\langle\bv_b,\nabla\bu\cdot\bn\rangle_{\partial T}=0$.
Using $q_0$ in $q=\{q_0,q_b\}\in W_h$ to test (\ref{ose2}), we arrive at
\begin{eqnarray*}
(\nabla\cdot\bu,q_0)=0.
\end{eqnarray*}
Using the fact that
$\sum_{T\in\mathcal{T}_h}\langle\bu\cdot\bn,q_b\rangle_{\partial
T}=0$ one has
\begin{equation}\label{eefc1}
\begin{aligned}
0=&(\nabla\cdot\bu,q_0)
\\
=&-(\bu,\nabla q_0)+\sum_{T\in\mathcal{T}_h}\langle\bu\cdot\bn,q_0\rangle_{\partial T}
\\
=&-(\bu,\nabla
q_0)+\sum_{T\in\mathcal{T}_h}\langle\bu\cdot\bn,q_0-q_b\rangle_{\partial
T}.
\end{aligned}
\end{equation}

Finally combining the equations (\ref{eefa1}) and (\ref{eefb1}) with (\ref{eefa2}) yields
\begin{equation}\label{eefl2}
\begin{aligned}
a(Q_h\bu,\bv)+b(\bv,\widetilde{Q}_hp)=&(\nabla_w Q_h\bu,\nabla_w\bv)+s(Q_h\bu,v)+(\bv_0,\widetilde{\nabla}_w\widetilde{Q}_hp)
\\
=&(\bf,\bv_0)+\sum_{T\in\mathcal{T}_h}\langle\bv_0-\bv_b,(\nabla\bu-\bQ_h\nabla\bu)\cdot\bn\rangle_{\partial T}
\\
&+s(Q_h\bu,\bv).
\end{aligned}
\end{equation}
Substituting it into (\ref{wf1}), then we would have
$$
a(\be_h,\bv)+b(\bv,\varepsilon_h)=s(Q_h\bu,\bv)+l_1(\bu,\bv).
$$
Combining the equations (\ref{eefb2}) with (\ref{eefc1}) we arrive at
\begin{equation}\label{eefl1}
\begin{aligned}
b(Q_h\bu,q)-c(\widetilde{Q}_hp,q)=&(\nabla q_0,\bu)-\sum_{T\in\mathcal{T}_h}\langle q_0-q_b,Q_0\bu\cdot\bn\rangle_{\partial T}-c(\widetilde{Q}_hp,q)
\\
=&\sum_{T\in\mathcal{T}_h}\langle\bu\cdot\bn,q_0-q_b\rangle_{\partial T}-\sum_{T\in\mathcal{T}_h}\langle q_0-q_b,Q_0\bu\cdot\bn\rangle_{\partial T}
\\
&-c(\widetilde{Q}_hp,q)
\\
=&\sum_{T\in\mathcal{T}_h}\langle
q_0-q_b,(\bu-Q_0\bu)\cdot\bn\rangle_{\partial T}
-c(\widetilde{Q}_hp,q).
\end{aligned}
\end{equation}
Substituting (\ref{eefl1}) into (\ref{wf2}) yields the following error equation
$$
b(\be_h,q)-c(\varepsilon_h,q)=l_2(\bu, q)-c(\widetilde{Q}_hp,q),
$$
for all $q\in W_h$, which completes the proof of (\ref{ee2}).
\end{proof}
\section{Error Estimates}
In this section we shall present the error estimates between the exact solution of (\ref{ose1})-(\ref{ose3}) and the numerical
solution of WG finite element method (\ref{wf1})-(\ref{wf2}). The two norms $\3bar\cdot\3bar$ and $\3bar\cdot\3bar_0$ are essentially $H^1$ norm and $L^2$ norm on $V_h$ and $W_h$ respectively. In this section we always assume $\mathcal{T}_h$ is shape regular {\rm(\cite{first})}.
\begin{theorem}
Let $(\bu,p)$ be the exact solution of (\ref{ose1})-(\ref{ose3}), $(\bu_h,p_h)$ be the numerical solution of (\ref{wf1})-(\ref{wf2}), then the following error estimates hold true
\begin{eqnarray}\label{ereq1}
\3bar\be_h\3bar+\3bar\varepsilon_h\3bar_*\le Ch^k(\|\bu\|_{k+1}+\|p\|_k),
\\
\|\varepsilon_0\|\le  Ch^k(\|\bu\|_{k+1}+\|p\|_k),
\end{eqnarray}
and consequently, one has
\begin{eqnarray}
\3bar \be_h\3bar+\3bar \varepsilon_h\3bar_0\le Ch^k(\|\bu\|_{k+1}+\|p\|_k).
\end{eqnarray}
\end{theorem}
\begin{proof}
Letting $\bv=\be_h$ in (\ref{ee1}) and $q=\varepsilon_h$ in (\ref{ee2}), we would obtain
\begin{equation*}
\begin{aligned}
\3bar\be_h\3bar^2+\3bar\varepsilon_h\3bar_*^2=&s(Q_h\bu,\be_h)+l_1(\bu,\be_h)
\\
&-l_2(\bu, \varepsilon_h)+c(\widetilde{Q}_hp,\varepsilon_h).
\end{aligned}
\end{equation*}
Then from (\ref{iqfs})-(\ref{iqfc}) we arrive at
$$
\3bar\be_h\3bar^2+\3bar\varepsilon_h\3bar_*^2\le Ch^k(\|\bu\|_{k+1}+\|p\|_k)(\3bar\be_h\3bar+\3bar\varepsilon_h\3bar_*),
$$
from which we would have
$$
\3bar\be_h\3bar+\3bar\varepsilon_h\3bar_*\le Ch^k(\|\bu\|_{k+1}+\|p\|_k).
$$

For any given $\rho\in W_h\subset L^2_0(\Omega)$, it follows from
{\rm\cite{WangYeStokes,rfb1,BF91,rfb3,GR86,Gunzburger89}} that there
is a $\widetilde{\bv}\in[H^1_0(\Omega)]^d$ such that
\begin{eqnarray}\label{btiqfb}
\frac{(\nabla\cdot\widetilde{\bv},\rho)}{\|\widetilde{\bv}\|_1}\ge C\|\rho\|,
\end{eqnarray}
where C is a positive constant which is dependent only on $\Omega$. Let $\bv=Q_h\widetilde{\bv}\in V_h$, we claim that the following inequality holds true
\begin{eqnarray}\label{ifv3}
\3bar\bv\3bar\le C\|\widetilde{\bv}\|_1,
\end{eqnarray}
where C is a constant.

From (\ref{qcfv}), we have
\begin{eqnarray*}
\sum_{T\in\mathcal{T}_h}\|\nabla_w\bv\|^2_T=\sum_{T\in\mathcal{T}_h}\|\nabla_w(Q_h\widetilde{\bv})\|^2_T=\sum_{T\in\mathcal{T}_h}\|\bQ_h(\nabla\widetilde{\bv})\|^2_T
\le\|\nabla\widetilde{\bv}\|^2.
\end{eqnarray*}
It follows from the definition of $Q_b$, (\ref{czi1}), and
(\ref{ti}) that
\begin{equation*}
\begin{aligned}
\sum_{T\in\mathcal{T}_h}h_T^{-1}\|\bv_0-\bv_b\|^2_{\partial T}&=\sum_{T\in\mathcal{T}_h}h_T^{-1}\|Q_0\widetilde{\bv}-Q_b \widetilde{\bv}\|^2_{\partial T}
\\
&=\sum_{T\in\mathcal{T}_h}h_T^{-1}\|Q_b(Q_0\widetilde{\bv})-Q_b\widetilde{\bv}\|^2_{\partial T}
\\
&\le\sum_{T\in\mathcal{T}_h}h_T^{-1}\|Q_0\widetilde{\bv}-\widetilde{\bv}\|^2_{\partial T}
\\
&\le C\|\nabla\widetilde{\bv}\|^2,
\end{aligned}
\end{equation*}
which yields
$$
\3bar\bv\3bar\le C\|\widetilde{\bv}\|_1.
$$

From (\ref{eefb2}), (\ref{btiqfb}), (\ref{1qfl3}), and the fact that
$\sum_{T\in\mathcal{T}_h}\langle
p_b,\widetilde{\bv}\cdot\bn\rangle_{\partial T}=0$ on
$\partial\Omega$, we would obtain
\begin{equation}\label{ifb}
\begin{aligned}
|b(\bv,p)|&=|b(Q_0\widetilde{\bv},p)|
\\
&=\left|(\nabla p_0,\widetilde{\bv})-\sum_{T\in\mathcal{T}_h}\langle p_0-p_b,Q_0\widetilde{\bv}\cdot\bn\rangle_{\partial T}\right|
\\
&=\left|-(p_0,\nabla\cdot\widetilde{\bv})-\sum_{T\in\mathcal{T}_h}\langle p_0-p_b,(Q_0\widetilde{\bv}-\widetilde{\bv})\cdot\bn\rangle_{\partial T}\right|
\\
&\ge
\left|(p_0,\nabla\cdot\widetilde{\bv})\right|-\left|\sum_{T\in\mathcal{T}_h}\langle
p_0-p_b,(Q_0\widetilde{\bv}-\widetilde{\bv})\cdot\bn\rangle_{\partial
T}\right|
\\
&\ge \|\widetilde{\bv}\|_1(C_1\|p_0\|-C_2\3bar p\3bar_*).
\end{aligned}
\end{equation}
Using (\ref{ee1}), (\ref{ereq1}), (\ref{iqfs}) and (\ref{iqfl1}), we obtain
\begin{equation}
\begin{aligned}\label{ifb1}
|b(\bv,\varepsilon_h)|&=|s(Q_h\bu,\bv)+l_1(\bu,\bv)-a(\be_h,\bv)|
\\
&\le Ch^k\|\bu\|_{k+1}\3bar\bv\3bar+\3bar\be_h\3bar\3bar\bv\3bar
\\
&\le Ch^k(\|\bu\|_{k+1}+\|p\|_k)\3bar\bv\3bar.
\end{aligned}
\end{equation}
Let $\bv$ be such that (\ref{ifb}) is true, it follows from (\ref{ifb1}) that
\begin{equation*}
\begin{aligned}
Ch^k(\|\bu\|_{k+1}+\|p\|_k)\geq\frac{|b(\bv,\varepsilon_h)|}{\3bar\bv\3bar}\geq C\frac{|b(\bv,\varepsilon_h)|}{\|\widetilde{\bv}\|_1}\geq C_1\|\varepsilon_0\|-C_2\3bar \varepsilon_h\3bar_*.
\end{aligned}
\end{equation*}
Then (\ref{ereq1}) implies that
$$
\|\varepsilon_0\|\le  Ch^k(\|\bu\|_{k+1}+\|p\|_k).
$$
From what we have demonstrated, one has
\begin{eqnarray*}
\3bar \be_h\3bar+\3bar \varepsilon_h\3bar_0\le Ch^k(\|\bu\|_{k+1}+\|p\|_k).
\end{eqnarray*}
\end{proof}

We shall use the dual technique to derive the $L^2$ error. Assume this problem has the $[H^2(\Omega)]^d\times H^1(\Omega)$-regularity, then the solution $(\psi;\xi)\in [H^2(\Omega)]^d\times H^1(\Omega)$ of the following equations
\begin{eqnarray}
\label{eefl3}-\Delta\psi+\nabla\xi&=&\be_0, \quad {\rm in}\ \Omega,
\\
\label{eefl4}\nabla\cdot\psi&=&0 ,~\quad {\rm in}\ \Omega,
\\
\label{eefl5}\psi&=&0 ,~\quad {\rm on}\ \partial\Omega
\end{eqnarray}
satisfies the following property
\begin{eqnarray}\label{eefl6}
\|\psi\|_2 + \|\xi\|_1\le C\|\be_0\|.
\end{eqnarray}
\begin{theorem}
Let $(\bu;p)\in[H^1(\Omega)\cap H^{k+1}(\Omega)]^d\times(L^2_0\cap H^k(\Omega))$
be the exact solution of (\ref{ose1})-(\ref{ose3}), $(\bu_h;p_h)\in V_h\times W_h$
be the numerical solution of (\ref{wf1})-(\ref{wf2}), $\be_0=Q_0\bu-\bu_0$ in
$\be=\{\be_0,\be_b\}$ then the following error estimate holds true
\begin{eqnarray}\label{eefl7}
\|\be_0\|\le Ch^{k+1}(\|\bu\|_{k+1}+\|p\|_k).
\end{eqnarray}
\end{theorem}
\begin{proof}
Since $(\psi;\xi)$ is the solution of (\ref{eefl3})-(\ref{eefl5}), letting $\bu=\psi, \bv=\be_h, p=\xi$ and $f=\be_0$ in (\ref{eefl2}) gives
\begin{eqnarray}\label{eefl8}
\|\be_0\|^2=a(Q_h\psi,\be_h)+b(\be_h,\widetilde{Q}_h\xi)-l_1(\psi,\be_h)-s(Q_h\psi,\be_h).
\end{eqnarray}
Letting $q=\widetilde{Q}_h\xi$ in (\ref{ee2}), we obtain
\begin{eqnarray}\label{eefl9}
b(\be_h,\widetilde{Q}_h\xi)=c(\varepsilon_h,\widetilde{Q}_h\xi)+l_2(\bu,\widetilde{Q}_h\xi)-c(\widetilde{Q}_h p,\widetilde{Q}_h\xi).
\end{eqnarray}
From (\ref{eefb2}) and (\ref{eefl4}), we arrive at
\begin{equation}\label{eefl10}
\begin{aligned}
b(Q_h\psi,\varepsilon_h)=& (\nabla\varepsilon_0,\psi)-\sum_{T\in\mathcal{T}_h}\langle\varepsilon_0-\varepsilon_b,Q_0\psi\cdot\bn\rangle_{\partial T}
\\
=&-(\varepsilon_0,\nabla\cdot\psi)+\sum_{\T\in\mathcal{T}_h}\langle\varepsilon_0,\psi\cdot\bn\rangle_{\partial T}-\sum_{T\in\mathcal{T}_h}\langle\varepsilon_0-\varepsilon_b,Q_0\psi\cdot\bn\rangle_{\partial T}
\\
=&-(\varepsilon_0,\nabla\cdot\psi)+\sum_{T\in\mathcal{T}_h}\langle\varepsilon_0-\varepsilon_b,(\psi-Q_0\psi)\cdot\bn\rangle_{\partial T}
\\
=&~l_2(\psi,\varepsilon_h),
\end{aligned}
\end{equation}
where we have used the fact that $\sum_{\T\in\mathcal{T}_h}\langle\varepsilon_b,\psi\cdot\bn\rangle_{\partial T}=0.$

Taking $\bv=Q_h\psi$ in (\ref{ee1}), combined with (\ref{eefl8})-(\ref{eefl10}) one has
\begin{equation}
\begin{aligned}
\label{bt1}\|\be_0\|^2=&s(Q_h\bu,Q_h\psi)+l_1(\bu,Q_h\psi)
\\
&-l_2(\psi, \varepsilon_h)+c(\widetilde{Q}_h\xi,\varepsilon_h)+l_2(\bu, \widetilde{Q}_h\xi)
\\
&-c(\widetilde{Q}_hp,\widetilde{Q}_h\xi)-l_1(\psi,\be_h)-s(Q_h\psi,\be_h).
\end{aligned}
\end{equation}
It follows from (\ref{iqfs})-(\ref{iqfl1}) that
\begin{equation}\label{bt2}
\begin{aligned}
|l_1(\psi,\be_h)+s(Q_h\psi,\be_h)|\le Ch(\|\psi\|_2+\|\xi\|_1)\3bar\be_h\3bar.
\end{aligned}
\end{equation}
Using (\ref{1qfl3})-(\ref{iqfc}), we would obtain
\begin{equation}\label{bt3}
\begin{aligned}
|l_2(\psi,\varepsilon_h)-c(\widetilde{Q}_h\xi,\varepsilon_h)|\le Ch(\|\psi\|_2+\|\xi\|_1)\3bar\varepsilon_h\3bar_*.
\end{aligned}
\end{equation}
From the definition of $Q_b$, (\ref{czi1}), (\ref{czi2}), and
(\ref{ti}), we arrive at
\begin{equation}\label{bt4}
\begin{aligned}
|l_1(\bu,Q_h\psi)|&=\left|\sum_{T\in\mathcal{T}_h}\langle Q_0\psi-Q_b\psi,(\nabla\bu-Q_h\nabla\bu)\cdot\bn\rangle_{\partial T}\right|
\\
&=\left|\sum_{T\in\mathcal{T}_h}\langle Q_b (Q_0\psi)-Q_b\psi,(\nabla\bu-Q_h\nabla\bu)\cdot\bn\rangle_{\partial T}\right|
\\
&=\left|\sum_{T\in\mathcal{T}_h}\langle Q_b(Q_0\psi-\psi),(\nabla\bu-Q_h\nabla\bu)\cdot\bn\rangle_{\partial T}\right|
\\
&\le\|Q_0\psi-\psi\|_{\partial T}\|\nabla\bu-\bQ_h\nabla\bu\|_{\partial T}
\\
&\le Ch^{k+1}\|\psi\|_2\|\bu\|_{k+1}.
\end{aligned}
\end{equation}
It follows from the definition of $Q_b$, (\ref{ti}), and (\ref{czi1}) that
\begin{equation}\label{bt5}
\begin{aligned}
|s(Q_h\bu,Q_h\psi)|=&\left|\sum_{T\in\mathcal{T}_h}h_T^{-1}\langle Q_0\bu-Q_b\bu,Q_0\psi-Q_b\psi\rangle_{\partial T}\right|
\\
=&\left|\sum_{T\in\mathcal{T}_h}h_T^{-1}\langle Q_b(Q_0\bu-\bu),Q_b(Q_0\psi-\psi)\rangle_{\partial T}\right|
\\
\le&\left(\sum_{T\in\mathcal{T}_h}h_T^{-1}\|Q_0\bu-\bu\|^2_{\partial T}\right)^{\frac12}\left(\sum_{T\in\mathcal{T}_h}h_T^{-1}\|Q_0\psi-\psi\|_{\partial T}\right)^{\frac12}
\\
\le&Ch^{k+1}\|\bu\|_{k+1}\|\psi\|_2.
\end{aligned}
\end{equation}
The definition of $Q_b$ together with (\ref{czi1}), (\ref{czi3}), and (\ref{ti}) yields
\begin{equation}\label{bt7}
\begin{aligned}
|l_2(\bu,\widetilde{Q}_h\xi)|=&\left|\sum_{T\in\mathcal{T}_h}\langle\widetilde{Q}_0\xi-\widetilde{Q}_b\xi,(\bu-Q_0\bu)\cdot\bn\rangle_{\partial T}\right|
\\
=&\left|\sum_{T\in\mathcal{T}_h}\langle\widetilde{Q}_b(\widetilde{Q}_0\xi)-\widetilde{Q}_b\xi,(\bu-Q_0\bu)\cdot\bn\rangle_{\partial T}\right|
\\
=&\left|\sum_{T\in\mathcal{T}_h}\langle\widetilde{Q}_b(\widetilde{Q}_0\xi-\xi),(\bu-Q_0\bu)\cdot\bn\rangle_{\partial T}\right|
\\
\le&\left(\sum_{T\in\mathcal{T}_h}\|\widetilde{Q}_0\xi-\xi\|^2_{\partial T}\right)^{\frac12} \left(\sum_{T\in\mathcal{T}_h}\|\bu-Q_0\bu\|^2_{\partial T}\right)^{\frac12}
\\
\le&Ch^{k+1}\|\xi\|_1\|\bu\|_{k+1}.
\end{aligned}
\end{equation}
From the definition of $Q_b$, (\ref{czi3}), and (\ref{ti}), we obtain
\begin{equation}\label{bt8}
\begin{aligned}
|c(\widetilde{Q}_hp,\widetilde{Q}_h\xi)|=&\left|\sum_{T\in\mathcal{T}_h}h_T\langle\widetilde{Q}_0p-\widetilde{Q}_bp,\widetilde{Q}_0\xi-\widetilde{Q}_b\xi\rangle_{\partial T}\right|
\\
=&\left|\sum_{T\in\mathcal{T}_h}h_T\langle\widetilde{Q}_b(\widetilde{Q}_0p-p),\widetilde{Q}_b(\widetilde{Q}_0\xi-\xi)\rangle_{\partial T}\right|
\\
\le&\left(\sum_{T\in\mathcal{T}_h}h_T\|\widetilde{Q}_0p-p\|^2_{\partial T}\right)^{\frac12} \left(\sum_{T\in\mathcal{T}_h}h_T\|\widetilde{Q}_0\xi-\xi\|^2_{\partial T}\right)^{\frac12}
\\
\le& Ch^{k+1}\|p\|_k\|\xi\|_1.
\end{aligned}
\end{equation}
From (\ref{bt1})-(\ref{bt8}), one has
$$
\|\be_0\|^2\le Ch^{k+1}(\|\psi\|_2+\|\xi\|_1)(\|\bu\|_{k+1}+\|p\|_k)+C h(\|\psi\|_2+\|\xi\|_1)(\3bar\be_h\3bar+\3bar\varepsilon_h\3bar_*),
$$
it follows from (\ref{eefl6}) that
$$
\|\be_0\|\le Ch^{k+1}(\|\bu\|_{k+1}+\|p\|_k)+Ch(\3bar\be_h\3bar+\3bar\varepsilon_h\3bar_*),
$$
together with (\ref{ereq1}), we would have
$$
\|\be_0\|\le C h^{k+1}(\|\bu\|_{k+1}+\|p\|_k),
$$
which completes the proof of the theorem.
\end{proof}

\section{Numerical Experiments}\label{Section:NumericalResults}

The goal of this section is to report some numerical results for the
weak Galerkin finite element method proposed and analyzed in
previous sections.

Let $(\bu;p)$ be the exact solution of (\ref{ose1})-(\ref{ose3}) and
$(\bu_h;p_h)$ be the numerical solution of (\ref{wf1})-(\ref{wf2}).
Denote $\be_h=Q_h\bu-\bu_h$ and
$\varepsilon_h=\widetilde{Q}_hp-p_h.$ The error for the weak
Galerkin solution is measured in four norms defined as follows:
\begin{eqnarray*}
\displaystyle|\!|\!| \be_h |\!|\!|^2&=&\sum_{T\in\mathcal{T}_h}
\left(\int_T|\nabla_w \be_h|^2 dT+ h_T^{-1}\int_{\partial
T}(\be_0-\be_b)^2 ds\right), \nonumber
\\[1mm]
\displaystyle\|\be_h\|^2&=&\sum_{T\in\mathcal{T}_h} \int_T|\be_h|^2
dT,
\\[1mm]
\displaystyle|\!|\!| \varepsilon_h |\!|\!|^2_0&=&
\sum_{T\in\mathcal{T}_h} \left(\int_T|\varepsilon_0|^2 dT +
  h_T\int_{\partial T}|\varepsilon_0- \varepsilon_b|^2
ds\right),
\\[1mm]
\displaystyle\|\varepsilon_h\|^2&=&\sum_{T\in\mathcal{T}_h}
\int_T|\varepsilon_h|^2 dT.
\end{eqnarray*}

\textbf{Example 7.1} {\rm Consider the problem
(\ref{ose1})-(\ref{ose3}) in the square domain $\Omega=(0,1)^2$. The
WG finite element space $k = 1$ is employed in the numerical
discretization. It has the analytic solution
\begin{eqnarray*}
\bu=\left(\begin{array}{c} \sin(\pi x) \sin(\pi y)
\\
\cos(\pi x) \cos(\pi y)
\end{array}\right) \ {\rm and\ } p=2 \cos(\pi x) \sin(\pi y).
\end{eqnarray*}
The right hand side function $\bf$ in (\ref{ose1}) is computed to
match the exact solution. The mesh size is denoted by $h$.

Table 7.1 shows that the errors and convergence rates of Example 7.1
in $|\!|\!| \cdot |\!|\!|-$ norm and $L^2-$norm for the WG-FEM
solution $\bu$ are of order $O(h)$ and $O(h^2)$ when $k=1$,
respectively.

Table 7.2 shows that the errors and orders of Example 7.1 in
$|\!|\!| \cdot |\!|\!|_0-$norm and $L^2-$norm for pressure when
$k=1$. The numerical results are also consistent with theory for
these two cases.

{\color{black}Table 7.3 and 7.4 show the errors and orders of Example
7.1 for the case $k=2$, and the convergence rates coincide with the
theoretical expectation.}

}

\begin{center}
Table 7.1. Numerical errors and orders for $\bu$ of Example 7.1 with $k=1$. \\
\begin{center}
    \begin{tabular}{ | c || c | c | c | c |}
    \hline
     $h$  & \  $|\!|\!| \be_h|\!|\!|$\  & \quad order \quad \,  &  \  $\|\be_h\|$  \   & \ \  order \ \  \\ \hline \hline
     1/4  & \  1.2347e+00  \   &          & \ 1.0681e-01 \ &           \\ \hline
     1/8  &  7.5411e-01    & 0.7113   & 2.8345e-02   & 1.9139     \\ \hline
    1/16  &  4.0953e-01    & 0.8808   & 7.8149e-03   & 1.8588     \\ \hline
    1/32  &  2.0483e-01    & 0.9995   & 2.0169e-03   & 1.9541     \\ \hline
    1/64  &  1.0172e-01    & 1.0099   & 5.0860e-04   & 1.9876     \\ \hline
    1/128 &  5.0471e-02    & 1.0110   & 1.2745e-04   & 1.9966     \\ \hline
    \hline
    \end{tabular}
\end{center}
\end{center}
\begin{center}
Table 7.2. Numerical errors and orders for $p$ of Example 7.1 with $k=1$. \\
\begin{center}
    \begin{tabular}{ | c || c | c | c | c |}
    \hline
     $h$  & \  $|\!|\!| \varepsilon_h|\!|\!|_0$\  & \quad order \quad \,  &  \  $\|\varepsilon_h\|$  \   & \ \  order \ \  \\ \hline \hline
     1/4  & \  1.1642e+00  \   &          & \ 7.6948e-01 \ &           \\ \hline
     1/8  &    5.1214e-01  & 1.1847   & 3.4266e-01   & 1.1671     \\ \hline
    1/16  &    2.1109e-01  & 1.2786   & 1.1063e-01   & 1.6311     \\ \hline
    1/32  &    9.3992e-02  & 1.1673   & 3.1403e-02   & 1.8168     \\ \hline
    1/64  &    4.4978e-02  & 1.0633   & 8.8252e-03   & 1.8312     \\ \hline
    1/128 &    2.2179e-02  & 1.0200   & 2.6757e-03   & 1.7217     \\ \hline
    \hline
    \end{tabular}
\end{center}
\end{center}

\begin{center}
Table 7.3. Numerical errors and orders for $\bu$ of Example 7.1 with $k=2$. \\
\begin{center}
    \begin{tabular}{ | c || c | c | c | c |}
    \hline
     $h$  & \  $|\!|\!| \be_h|\!|\!|$\  & \quad order \quad \,  &  \  $\|\be_h\|$  \   & \ \  order \ \  \\ \hline \hline
   1/4   & 2.3130e-01 &         & 1.8281e-02 &         \\ \hline
   1/8   & 5.9891e-02 & 1.9494  & 2.3513e-03 & 2.9588  \\ \hline
   1/16  & 1.4832e-02 & 2.0136  & 2.9732e-04 & 2.9834  \\ \hline
   1/32  & 3.7134e-03 & 1.9979  & 3.7349e-05 & 2.9929  \\ \hline
   1/64  & 9.2987e-04 & 1.9977  & 4.6792e-06 & 2.9967  \\ \hline
   1/128 & 2.3265e-04 & 1.9989  & 5.8553e-07 & 2.9984  \\ \hline
    \hline
    \end{tabular}
\end{center}
\end{center}

\begin{center}
Table 7.4. Numerical errors and orders for $p$ of Example 7.1 with $k=2$. \\
\begin{center}
    \begin{tabular}{ | c || c | c | c | c |}
    \hline
     $h$  & \  $|\!|\!| \varepsilon_h|\!|\!|_0$\  & \quad order \quad \,  &  \  $\|\varepsilon_h\|$  \   & \ \  order \ \  \\ \hline \hline
   1/4   & 1.9753e-01 &         & 4.3517e-02 &         \\ \hline
   1/8   & 5.0278e-02 & 1.9741  & 8.5422e-03 & 2.3489  \\ \hline
   1/16  & 1.2600e-02 & 1.9965  & 1.8980e-03 & 2.1702  \\ \hline
   1/32  & 3.1485e-03 & 2.0007  & 4.5063e-04 & 2.0744  \\ \hline
   1/64  & 7.8661e-04 & 2.0010  & 1.1009e-04 & 2.0333  \\ \hline
   1/128 & 1.9657e-04 & 2.0006  & 2.7227e-05 & 2.0155  \\ \hline
\hline
    \end{tabular}
\end{center}
\end{center}

\textbf{Example 7.2} {\rm Consider the problem
(\ref{ose1})-(\ref{ose3}) in the square domain $\Omega=(0,1)^2$. The
WG finite element space $k = 1$ is employed in the numerical
discretization. It has the analytic solution
\begin{eqnarray*}
\bu=\left(\begin{array}{c} 2\pi\sin^2(\pi x)\cos(\pi y)
\sin(\pi y)
\\
-2\pi\sin(\pi x) \cos(\pi x) \sin^2(\pi y)
\end{array}\right)
\end{eqnarray*}
and
$$
p=\cos(\pi x) \cos(\pi y).
$$
The right hand side function $\bf$ in (\ref{ose1}) is computed to
match the exact solution. The mesh size is denoted by $h$.

The numerical results are presented in Tables 7.5-7.8, which confirm
the theory developed in previous sections. }

\begin{center}
Table 7.5. Numerical errors and orders for $\bu$ of Example 7.2 with $k=1$. \\
\begin{center}
    \begin{tabular}{ | c || c | c | c | c |}
    \hline
     $h$  & \  $|\!|\!| \be_h|\!|\!|$\  & \quad order \quad \,  &  \  $\|\be_h\|$  \   & \ \  order \ \  \\ \hline \hline
     1/4  & \  1.3024e+01  \   &          & \ 1.9402e+00 \ &           \\ \hline
     1/8  &  6.2924e+00    & 1.0494   & 3.1369e-01   & 2.6288     \\ \hline
    1/16  &  3.1404e+00    & 1.0027   & 5.7291e-02   & 2.4530     \\ \hline
    1/32  &  1.5840e+00    & 0.9874   & 1.2695e-02   & 2.1740     \\ \hline
    1/64  &  7.9695e-01    & 0.9910   & 3.0804e-03   & 2.0431     \\ \hline
    1/128 &  3.9961e-01    & 0.9959   & 7.6596e-04   & 2.0078     \\ \hline
    \hline
    \end{tabular}
\end{center}
\end{center}

\begin{center}
Table 7.6. Numerical errors and orders for $p$ of Example 7.2 with $k=1$. \\
\begin{center}
    \begin{tabular}{ | c || c | c | c | c |}
    \hline
     $h$  & \  $|\!|\!| \varepsilon_h|\!|\!|_0$\  & \quad order \quad \,  &  \  $\|\varepsilon_h\|$  \   & \ \  order \ \  \\ \hline \hline
     1/4  & \  2.5875e+00  \   &          & \ 6.7394e-01 \ &           \\ \hline
     1/8  &    1.1518e+00  & 1.1676   & 5.3190e-01   & 0.3415     \\ \hline
    1/16  &    5.1671e-01  & 1.1565   & 2.8686e-01   & 0.8908     \\ \hline
    1/32  &    2.2432e-01  & 1.2038   & 1.1558e-01   & 1.3114     \\ \hline
    1/64  &    9.9999e-02  & 1.1656   & 4.1483e-02   & 1.4783     \\ \hline
    1/128 &    4.6840e-02  & 1.0942   & 1.5196e-02   & 1.4489     \\ \hline
    \hline
    \end{tabular}
\end{center}
\end{center}

\begin{center}
Table 7.7. Numerical errors and orders for $\bu$ of Example 7.2 with $k=2$. \\
\begin{center}
    \begin{tabular}{ | c || c | c | c | c |}
    \hline
     $h$  & \  $|\!|\!| \be_h|\!|\!|$\  & \quad order \quad \,  &  \  $\|\be_h\|$  \   & \ \  order \ \  \\ \hline \hline
   1/4   & 3.1376e+00 &         & 2.4616e-01 &         \\ \hline
   1/8   & 8.2291e-01 & 1.9309  & 3.2044e-02 & 2.9415  \\ \hline
   1/16  & 2.0280e-01 & 2.0207  & 4.0772e-03 & 2.9744  \\ \hline
   1/32  & 5.0679e-02 & 2.0006  & 5.1315e-04 & 2.9901  \\ \hline
   1/64  & 1.2698e-02 & 1.9968  & 6.4328e-05 & 2.9959  \\ \hline
   1/128 & 3.1774e-03 & 1.9986  & 8.0513e-06 & 2.9981  \\ \hline
\hline
    \end{tabular}
\end{center}
\end{center}

\begin{center}
Table 7.8. Numerical errors and orders for $p$ of Example 7.2 with $k=2$. \\
\begin{center}
    \begin{tabular}{ | c || c | c | c | c |}
    \hline
     $h$  & \  $|\!|\!| \varepsilon_h|\!|\!|_0$\  & \quad order \quad \,  &  \  $\|\varepsilon_h\|$  \   & \ \  order \ \  \\ \hline \hline
   1/4   & 1.3811e+00 &         & 4.8957e-01 &         \\ \hline
   1/8   & 3.5346e-01 & 1.9662  & 8.5399e-02 & 2.5192  \\ \hline
   1/16  & 8.7131e-02 & 2.0203  & 1.5911e-02 & 2.4242  \\ \hline
   1/32  & 2.1567e-02 & 2.0144  & 3.4240e-03 & 2.2163  \\ \hline
   1/64  & 5.3636e-03 & 2.0075  & 7.9926e-04 & 2.0990  \\ \hline
   1/128 & 1.3374e-03 & 2.0038  & 1.9366e-04 & 2.0452  \\ \hline
\hline
    \end{tabular}
\end{center}
\end{center}

\textbf{Example 7.3} Consider the following lid-driven
cavity problem. in the square domain $\Omega=(0,1)^2$.
{\color{black}
This is a benchmark testcase for Stokes flow,
which has been tested in \cite{ESW05,Liu11,WangWangYe09,YangLiuLin15}.
A delicate analysis of solution regularity is presented in
\cite{CaiWang09}.
}

In this
example, a uniform mesh with step $h=1/32$ and polynomial degree
$k=2$ are employed. The source term in (\ref{ose1}) is
$\textbf{f}=0$ and the Dirichlet boundary condition is given as
\begin{eqnarray*}
\textbf{u}= \left\{
\begin{array}{l} (1,0)^T,\quad \text{if }x=1, y\in(0,1),
\\
0,\quad \text{otherwise}.
\end{array}
\right.
\end{eqnarray*}
The exact solution of lid-driven cavity problem is unknown, which
has singularity at point $(1,0)$ and $(1,1)$.

The vectograph and streamlines of the velocity field
are presented in Fig 7.1 and Fig7.2. The shape of streamlines is
similar to the result obtained by IFISS{\color{black}\cite{ERS07}}.

%

\begin{figure}
\begin{center}
   \begin{tabular}{ccc}
     \resizebox{5cm}{4.5cm}{\includegraphics{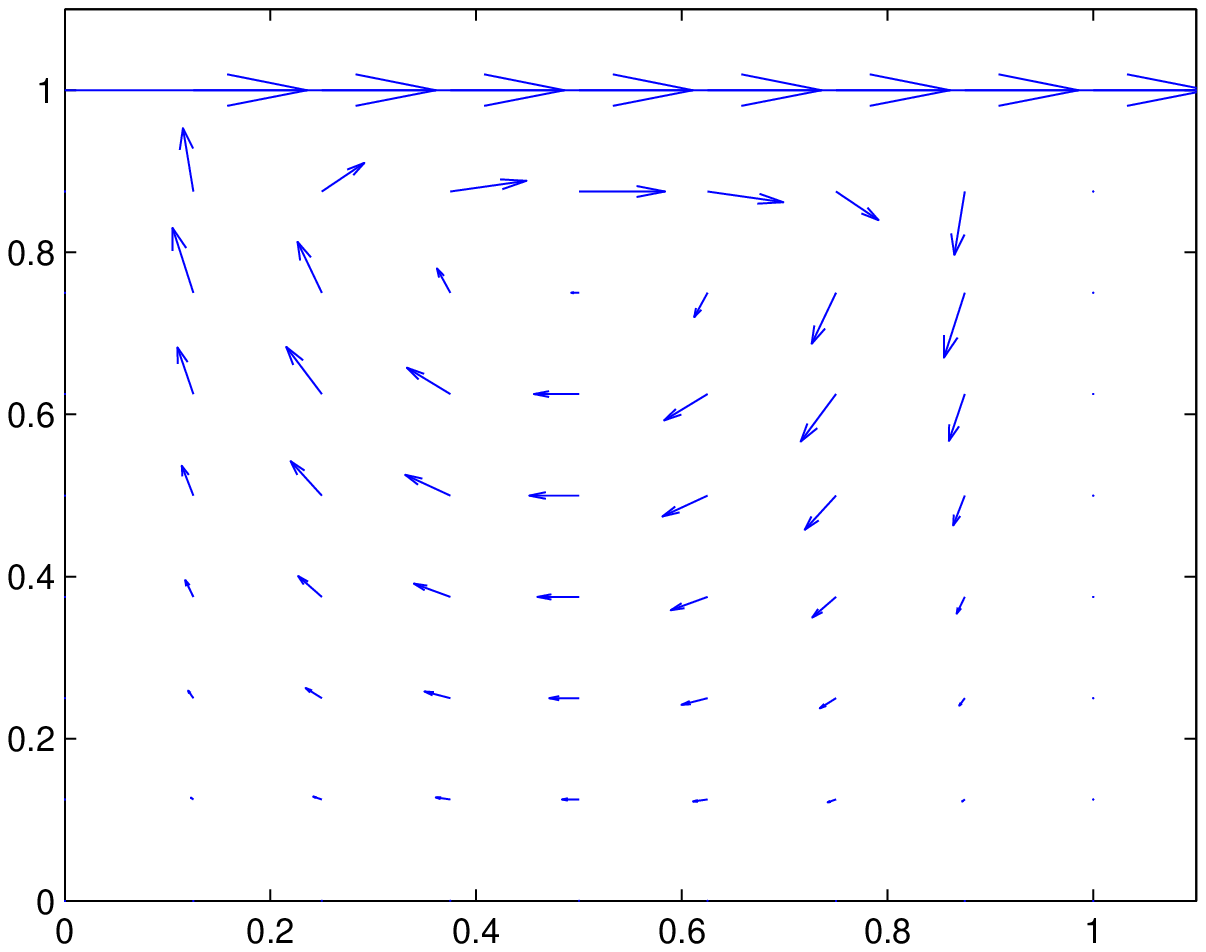}}
     &&
     \resizebox{5cm}{4.5cm}{\includegraphics{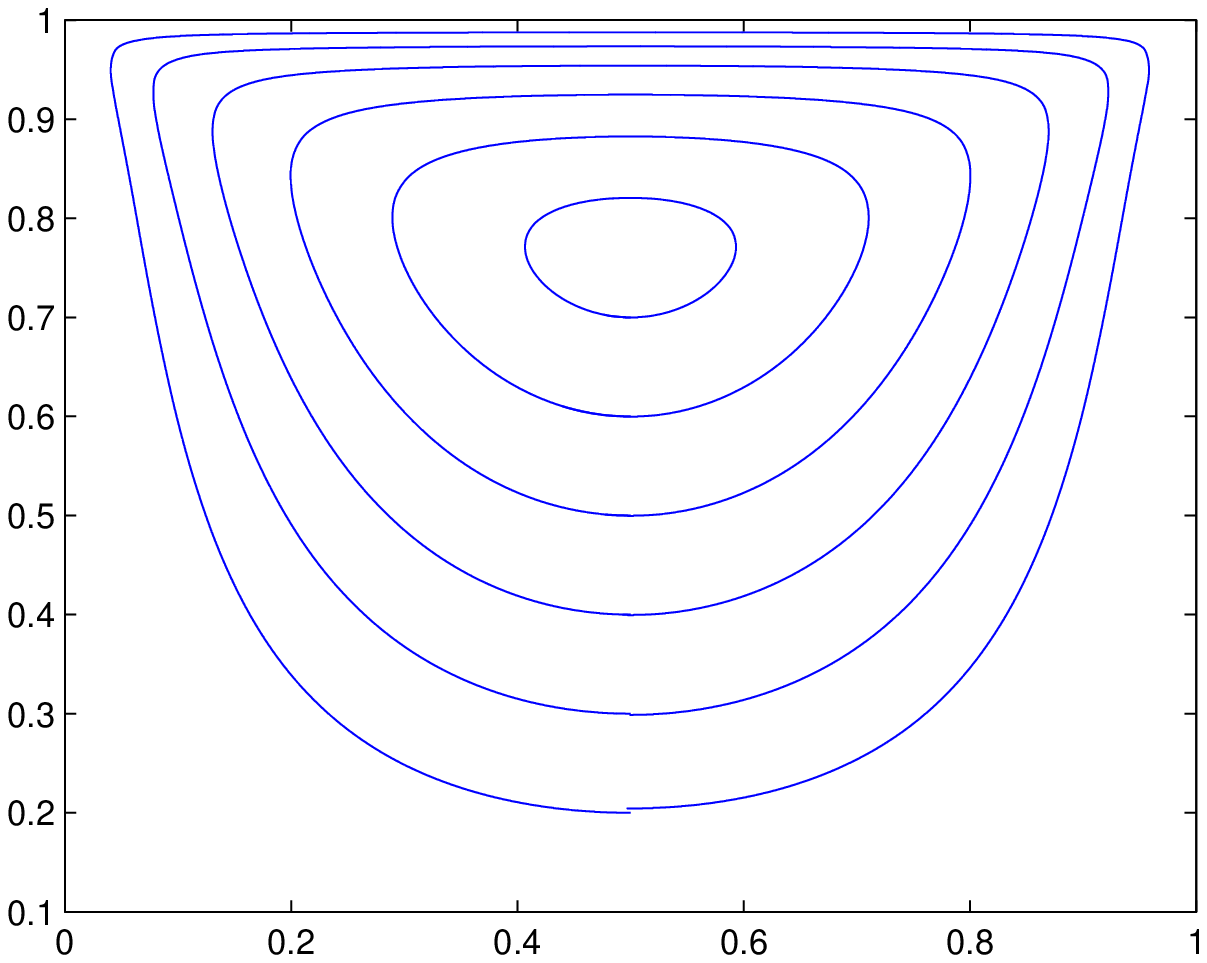}}
   \end{tabular}
   \caption{Left panel: The vectorgraph of velocity; Right panel:  The streamlines of velocity}
\end{center}
\end{figure}

\section{Appendix}

In this section, we will give some important inequalities.

\begin{lemma}{\rm(\cite{first,WangYe2014})}
Assume that $\mathcal{T}_h$, the finite element partition of $\Omega$, is shape regular. Let $\bu\in [H^{r+1}(\Omega)]^d$ and $p\in H^r(\Omega)$ with $1\le r\le k$. Then for any $0\le s\le 1$ we have
\begin{eqnarray}
\label{czi1}&\sum_{T\in\mathcal{T}_h}&h_T^{2s}\|\bu-Q_0\bu\|^2_{T,s}\le Ch^{2(r+1)}\|\bu\|^2_{r+1},
\\
\label{czi2}&\sum_{T\in\mathcal{T}_h}&h_T^{2s}\|\nabla\bu-\bQ(\nabla\bu)\|^2_{T,s}\le Ch^{2r}\|\bu\|^2_{r+1},
\\
\label{czi3}&\sum_{T\in\mathcal{T}_h}&h_T^{2s}\|p-\widetilde{Q}_0 p\|^2_{T,s}\le Ch^{2r}\|p\|^2_r,
\end{eqnarray}
where $C$ is a constant which is independent of the meshsize h and the functions.
\end{lemma}
\begin{lemma}{\rm(\cite{WangYe2014})}
Let T be an element of the finite element partition $\mathcal{T}_h$,
$e$ is an edge or face which is part of $\partial T$. For any
function $g\in H^1(T)$, the following trace inequality holds true
\begin{eqnarray}
\label{ti}\|g\|_e^2\le C(h_T^{-1}\|g\|^2_T+h_T\|\nabla g\|^2_T).
\end{eqnarray}
\end{lemma}
\begin{lemma}
Let $\bu\in[H_0^{k+1}(\Omega)]^d$ , $p\in H^k(\Omega)$, $\bv\in V_h$
and $q\in W_h$. The following estimates hold true
\begin{eqnarray}
\label{iqfs}|s(Q_h\bu,\bv)|&\le& Ch^k\|\bu\|_{k+1}\3bar\bv\3bar,
\\
\label{iqfl1}|l_1(\bu,\bv)|&\le&Ch^k\|\bu\|_{k+1}\3bar\bv\3bar,
\\
\label{1qfl3}|l_2(\bu, q)|&\le&Ch^k\|\bu\|_{k+1}\3barq\3bar_*,
\\
\label{iqfc}|c(\widetilde{Q}_hp,q)|&\le&Ch^k\|p\|_{k}\3barq\3bar_*.
\end{eqnarray}
\end{lemma}
\begin{proof}
It follows from the definition of $Q_b$ , (\ref{czi1}), and (\ref{ti}) that
\begin{equation*}
\begin{aligned}
|s(Q_h\bu,\bv)|&=\left|\sum_{T\in\mathcal{T}_h}h^{-1}_T(Q_0\bu-Q_b\bu,\bv_0-\bv_b)_{\partial T}\right|
\\
&\le \left(\sum_{T\in\mathcal{T}_h}h^{-1}_T\|Q_0\bu-\bu\|_{\partial T}^2\right)^{\frac12}\left(\sum_{T\in\mathcal{T}_h}h^{-1}_T\|\bv_0-\bv_b\|_{\partial T}^2\right)^{\frac12}
\\
&\le Ch^k\|\bu\|_{k+1}\3bar\bv\3bar.
\end{aligned}
\end{equation*}
From (\ref{czi2}) and (\ref{ti}), we would have
\begin{equation*}
\begin{aligned}
|l_1(\bu,\bv)|&=\left|\sum_{T\in\mathcal{T}_h}\langle\bv_0-\bv_b,(\nabla\bu-\bQ_h\nabla\bu)\cdot\bn\rangle_{\partial T}\right|
\\
&\le\left(\sum_{T\in\mathcal{T}_h}h^{-1}_T\|\bv_0-\bv_b\|_{\partial T}^2\right)^{\frac12}\left(\sum_{T\in\mathcal{T}_h}h_T\|\nabla\bu-\bQ_h\nabla\bu\|_{\partial T}^2\right)^{\frac12}
\\
&\le Ch^k\|\bu\|_{k+1}\3bar\bv\3bar.
\end{aligned}
\end{equation*}
Using (\ref{czi1}) and (\ref{ti}), we would arrive at
\begin{equation*}
\begin{aligned}
|l_2(\bu, q)|&=\left|\sum_{T\in\mathcal{T}_h}\langle q_0-q_b,(\bu-Q_0\bu)\cdot\bn\rangle_{\partial T}\right|
\\
&\le\left(\sum_{T\in\mathcal{T}_h}h_T\|q_0-q_b\|_{\partial T}^2\right)^{\frac12}\left(\sum_{T\in\mathcal{T}_h}h^{-1}_T\|Q_0\bu-\bu\|_{\partial T}^2\right)^{\frac12}
\\
&\le Ch^k\|u\|_{k+1}\3bar q\3bar_*.
\end{aligned}
\end{equation*}
From the definition of $Q_b$ and (\ref{czi3}), we obtain
\begin{equation*}
\begin{aligned}
|c(\widetilde{Q}_hp,q)|&=\left|\sum_{T\in\mathcal{T}_h}h_T\langle\widetilde{Q}_0p-\widetilde{Q}_bp,q_0-q_b\rangle_{\partial T}\right|
\\
&=\left|\sum_{T\in\mathcal{T}_h}h_T\langle\widetilde{Q}_0p-p,q_0-q_b\rangle_{\partial T}\right|
\\
&\le\left(\sum_{T\in\mathcal{T}_h}h_T\|\widetilde{Q}_0p-p\|_{\partial T}^2\right)^{\frac12}\left(\sum_{T\in\mathcal{T}_h}h_T\|q_0-q_b\|_{\partial T}^2\right)^{\frac12}
\\
&\le Ch^k\|p\|_{k}\3bar q\3bar_*.
\end{aligned}
\end{equation*}
Now we have proved all the estimates in this lemma.
\end{proof}

\end{document}